\documentclass[12pt]{article}
\usepackage[]{amsmath,amssymb}
\usepackage{amscd}
\usepackage{latexsym}
\usepackage{cite}
\usepackage{amsthm}

\newtheorem{definition}{Definition}[section]
\newtheorem{theorem}[definition]{Theorem}
\newtheorem{lemma}[definition]{Lemma}
\newtheorem{corollary}[definition]{Corollary}

\newtheorem{example}[definition]{Example}

\newtheorem{problem}[definition]{Problem}
\newtheorem{note}[definition]{Note}

\newtheorem{proposition}[definition]{Proposition}

\begin{document}
\title{\bf  
The Lusztig automorphism
of the
\\
$q$-Onsager 
algebra
}
%Tatsuro Ito\footnote{Supported in part by JSPS grant
%18340022.} $\;$   and
\author{
Paul Terwilliger 
}
\date{}
%\footnote{This author gratefully acknowledges 
%support from the FY2007 JSPS Invitation Fellowship Program
%for Reseach in Japan (Long-Term), grant L-07512.}
%}
%\date{}
%to get date printout, comment out above line

\maketitle
\begin{abstract}
Pascal Baseilhac and Stefan Kolb recently introduced the Lusztig automorphism
$L$ of the $q$-Onsager algebra $\mathcal O_q$.
In this paper, we express each of $L, L^{-1}$ as a formal sum
involving some quantum adjoints. 
In addition,
(i) we give a computer-free proof that $L$ exists;
(ii) we establish the higher order $q$-Dolan/Grady relations previously
conjectured by Baseilhac and Thao Vu; (iii) we obtain a Lusztig
automorphism for the current algebra $\mathcal A_q$ associated with
$\mathcal O_q$; (iv) we describe what happens
when a finite-dimensional irreducible $\mathcal O_q$-module
is twisted via $L$.


 \end{abstract}

\section{Introduction}

\noindent Throughout this paper $\mathbb F$ denotes a field.
Fix $0 \not=q \in \mathbb F$ that is not a root of unity. Recall 
the notation
\begin{eqnarray*}
\lbrack n\rbrack_q = \frac{q^n-q^{-n}}{q-q^{-1}}
\qquad \qquad n = 0,1,2,\ldots
\end{eqnarray*}
We will be discussing algebras. An algebra is meant to 
be associative and have a $1$. A subalgebra has the same $1$
as the parent algebra.

\begin{definition}
\label{def:qons}
{\rm (See \cite[Section~2]{bas1},
\cite[Definition~3.9]{qSerre}.)}
\rm Let ${\mathcal O}_q$ denote the 
$\mathbb F$-algebra with generators $A,B$ and relations
\begin{eqnarray}
&&A^3B-\lbrack 3\rbrack_q A^2BA+ 
\lbrack 3\rbrack_q ABA^2 -BA^3 = (q^2-q^{-2})^2 (BA-AB),
\label{eq:dg1}
\\
&&B^3A-\lbrack 3\rbrack_q B^2AB + 
\lbrack 3\rbrack_q BAB^2 -AB^3 = (q^2-q^{-2})^2 (AB-BA).
\label{eq:dg2}
\end{eqnarray}
\noindent We call ${\mathcal O}_q$ the {\it  $q$-Onsager algebra}. 
The relations (\ref{eq:dg1}),  (\ref{eq:dg2}) 
are called the {\it $q$-Dolan/Grady relations}.
\end{definition}
\noindent We now give some background on $\mathcal O_q$; for more
information see
\cite{madrid}.
There is a family of algebras called tridiagonal algebras
\cite[Definition~3.9]{qSerre} that arise in the study of
($P$ and $Q$)-polynomial association schemes
\cite[Lemma~5.4]{tersub3} and tridiagonal pairs
\cite[Theorem~10.1]{TD00},
\cite[Theorem~3.10]{qSerre}.
The algebra $\mathcal O_q$ is the ``most general'' example
of a tridiagonal algebra
\cite[p.~70]{ItoTer}.
Applications of $\mathcal O_q$ to 
tridiagonal pairs can be found in
    \cite{bas6,  INT,
TD00, ItoTer,
IT:aug,
LS99,
 qSerre,
aw}.
The algebra $\mathcal O_q$ has applications
to quantum integrable models
\cite{bas2,
bas1,
basDef,
    bas6,
 bas8,
basXXZ,
basBel,
BK05,
bas4,
  basKoi},
%basnc,
%BVu},
reflection equation algebras
\cite{basnc},
and coideal subalgebras
\cite{bc,kolb, kolb1}.
There is an algebra homomorphism from $\mathcal O_q$ into
the algebra $\square_q$ 
\cite[Proposition~5.6]{pospart},
and the universal Askey-Wilson algebra
\cite[Sections~9,10]{uaw}.
\medskip

%%%%%%%%%%%%%%%%
%\begin{note}\rm A tridiagonal pair of $q$-Racah type is essentially
%the same thing as a finite dimensional irreducible
%${\mathcal O}_q$-module.
%\end{note}
%%%%%%%%%%%

\noindent In \cite{BK} Pascal Baseilhac and Stefan Kolb
found an automorphism $L$ of ${\mathcal O}_q$
that acts as follows:
\begin{eqnarray}
&&L(A)=A, \qquad \qquad
L(B) = B + \frac{qA^2B-(q+q^{-1})ABA+q^{-1}BA^2}{(q-q^{-1})(q^2-q^{-2})},
\label{eq:bk1}
\\
&&L^{-1}(A)=A, \qquad \qquad
L^{-1}(B) = B + \frac{q^{-1}A^2B-(q+q^{-1})ABA+qBA^2}{(q-q^{-1})(q^2-q^{-2})}.
\label{eq:bk2}
\end{eqnarray}
They called $L$ the {\it Lusztig automorphism}
of ${\mathcal O}_q$. In our view $L$ is a profound
discovery, and worthy of much further study.
In this paper,
we express each of $L$, $L^{-1}$ as a formal sum that involves
some quantum adjoints of $A$. In addition,
(i) we obtain 
 a computer-free proof that $L$ exists; (ii)
 we establish the higher order $q$-Dolan/Grady relations
 previously conjectured by Baseilhac and Thao Vu \cite{BVu};
 (iii) we obtain a Lusztig automorphism for the current algebra
$\mathcal A_q$ \cite[Definition 3.1]{basnc} associated with $\mathcal O_q$;
(iv) we describe what happens when a finite-dimensional irreducible
$\mathcal O_q$-module is twisted via $L$.

\section{Statement of the main result}
We will state our main result after a few comments.
Recall the natural numbers $\mathbb N=\lbrace 0,1,2,\ldots\rbrace$
and integers
 $\mathbb Z=\lbrace 0,\pm 1,\pm 2,\ldots\rbrace$.
Let $\mathcal A$ denote an $\mathbb F$-algebra.
For $A \in \mathcal A$,
the corresponding adjoint map is
${\rm ad}\, A: \mathcal A \to \mathcal A$,
$X \mapsto A X - X A$.
For $r \in \mathbb Z$ define the quantum adjoint map 
${\rm ad}_r A: \mathcal A \to \mathcal A$,
$X \mapsto q^r A X - q^{-r} X A$.
We have ${\rm ad}\,A = {\rm ad}_0 A$.
Note that
${\rm ad}_r A$, 
${\rm ad}_s A$ commute for $r,s \in \mathbb Z$.
\medskip

\noindent
We now state our main result.

\begin{theorem}
\label{conj:T}
The Lusztig automorphism $L$ of $\mathcal O_q$
satisfies
\begin{eqnarray}
&&L = I + \sum_{n=1}^\infty
\Biggl(
\frac{{\rm ad}\,A}{q-q^{-1}}
\prod_{r=1}^{n-1} \frac{(q^{2r}-q^{-2r})^2 I +
({\rm ad}_r A )({\rm ad}_{-r}A)}{(q^{2r}-q^{-2r})(q^{2r+1}-q^{-2r-1})}\Biggr) 
\frac{{\rm ad}_n A}{q^{2n}-q^{-2n}},
\label{eq:conj1}
\\
&&L^{-1} = I + \sum_{n=1}^\infty 
\Biggl(
\frac{{\rm ad}\,A}{q-q^{-1}}
\prod_{r=1}^{n-1} \frac{(q^{2r}-q^{-2r})^2 I +
({\rm ad}_r A )({\rm ad}_{-r} A)}{(q^{2r}-q^{-2r})(q^{2r+1}-q^{-2r-1})}\Biggr) 
\frac{{\rm ad}_{-n} A}{q^{2n}-q^{-2n}}.
\label{eq:conj2}
\end{eqnarray}
Moreover for all $X \in \mathcal O_q$, in the above sums 
the large parenthetical expression 
vanishes at $X$ for all but finitely $n$.
\end{theorem}

\noindent  We mention two consequences of Theorem
\ref{conj:T}.

\begin{corollary}
\label{conj:comA}
The automorphism $L$ fixes every
element of $\mathcal O_q$ that commutes with $A$.
\end{corollary}

\begin{corollary}
\label{conj:com1A}
Pick $X \in \mathcal O_q$ such that
\begin{eqnarray*}
A^3X-\lbrack 3\rbrack_q A^2XA+ 
\lbrack 3\rbrack_q AXA^2 -XA^3 = (q^2-q^{-2})^2 (XA-AX).
\end{eqnarray*}
Then $L$ sends
\begin{eqnarray*}
X \mapsto 
 X + \frac{qA^2X-(q+q^{-1})AXA+q^{-1}XA^2}{(q-q^{-1})(q^2-q^{-2})},
 \end{eqnarray*}
 and $L^{-1}$ sends
\begin{eqnarray*}
X \mapsto 
 X + \frac{q^{-1}A^2X-(q+q^{-1})AXA+qXA^2}{(q-q^{-1})(q^2-q^{-2})}.
 \end{eqnarray*}
\end{corollary}

\noindent We will obtain Theorem
 \ref{conj:T} as a consequence of a more general result,
 which we now summarize.
  Let $\mathcal A$ denote an $\mathbb F$-algebra
 and let $A \in \mathcal A$. Consider the formal sums
\begin{eqnarray}
&&S = I + \sum_{n=1}^\infty
\Biggl(
\frac{{\rm ad}\,A}{q-q^{-1}}
\prod_{r=1}^{n-1} \frac{(q^{2r}-q^{-2r})^2 I +
({\rm ad}_r A )({\rm ad}_{-r}A)}{(q^{2r}-q^{-2r})(q^{2r+1}-q^{-2r-1})}\Biggr) 
\frac{{\rm ad}_n A}{q^{2n}-q^{-2n}},
\label{eq:conj1S}
\\
&&S' = I + \sum_{n=1}^\infty 
\Biggl(
\frac{{\rm ad}\,A}{q-q^{-1}}
\prod_{r=1}^{n-1} \frac{(q^{2r}-q^{-2r})^2 I +
({\rm ad}_r A )({\rm ad}_{-r} A)}{(q^{2r}-q^{-2r})(q^{2r+1}-q^{-2r-1})}\Biggr) 
\frac{{\rm ad}_{-n} A}{q^{2n}-q^{-2n}}.
\label{eq:conj2S}
\end{eqnarray}
 An element $X \in \mathcal A$ is called
 $A$-standard 
  whenever 
 the large parenthetical expression in
(\ref{eq:conj1S}),
(\ref{eq:conj2S})
 vanishes at $X$ for all but finitely many $n$.
The algebra $\mathcal A$ is called $A$-standard whenever each element
of $\mathcal A$ is $A$-standard.
Assume that $\mathcal A$ is $A$-standard.
We will show that $S$ and $S'$ act on $\mathcal A$
as an automorphism, and these automorphisms are inverses.
Also, we will show that the algebra 
$\mathcal O_q$ is $A$-standard and
$S=L$, $S'=L^{-1}$.

\section{Some identities for the quantum adjoint}

As we work towards Theorem
\ref{conj:T},
 our first goal is to establish some identities
for the quantum adjoint, that apply to any
$\mathbb F$-algebra. Let $\mathcal A$ denote an $\mathbb F$-algebra,
and fix $A \in \mathcal A$.
Recall the sums $S$, $S'$ from
(\ref{eq:conj1S}),
(\ref{eq:conj2S}).
 We will be discussing the terms in
 these sums.
To simplify this discussion we
  introduce 
a ``balanced'' version of $\rm ad$, called $\rm bad$.
Let $\mathbb Z^+$ denote the set of positive integers.

\begin{definition}
\label{def:bad}
\rm
Define
\begin{eqnarray*}
{\rm bad}_0 A = \frac{  {\rm ad}\,A}{q-q^{-1}}
\end{eqnarray*}
and 
\begin{eqnarray*}
{\rm bad}_n A = \frac{(q^{2n}-q^{-2n})^2 I + ({\rm ad}_n A)
({\rm ad}_{-n} A)}{(q^{2n}-q^{-2n})(q^{2n+1}-q^{-2n-1})}
\qquad \qquad n\in \mathbb Z^+.
\end{eqnarray*}
Further define
\begin{eqnarray}
\label{eq:bad}
({\rm bad}\,A)_n = \prod_{i=0}^{n-1} {\rm bad}_i A 
\qquad \qquad n\in \mathbb N.
\end{eqnarray}
We interpret 
$({\rm bad}\,A)_0 = I$.
\end{definition}

\begin{definition}
\label{def:Tn}
\rm
Define
$S_0=I$ and
\begin{eqnarray*}
S_n = \frac{
({\rm bad}\, A)_n
 \,{\rm ad}_{n} A}{q^{2n}-q^{-2n}} 
\qquad \qquad 
n\in \mathbb Z^+.
\end{eqnarray*}
Further define
$S'_0=I$ and
\begin{eqnarray*}
S'_n = \frac{
({\rm bad}\, A)_n
 \,{\rm ad}_{-n} A}{q^{2n}-q^{-2n}} 
\qquad \qquad 
n\in \mathbb Z^+.
\end{eqnarray*}
\end{definition}

\begin{lemma} 
\label{def:fs}
In the above notation the sums
{\rm (\ref{eq:conj1S})},
{\rm (\ref{eq:conj2S})} become
\begin{eqnarray*}
S = \sum_{n \in \mathbb N} S_n,
\qquad \qquad 
S' = \sum_{n\in \mathbb N} S'_n.
\end{eqnarray*}
\end{lemma}

\noindent Our next goal is to prove Proposition
\ref{prop:TTp} below.
To this end we give some identities that hold in $\mathcal A$.

\begin{lemma}
\label{lem:plus}
For $i \in \mathbb Z$,
\begin{eqnarray*}
 {\rm ad}_{i} A + 
 {\rm ad}_{-i} A  = (q^i+q^{-i}) 
 \,{\rm ad}\, A.
\end{eqnarray*}
\end{lemma}
\begin{proof} Routine.
\end{proof}

\begin{lemma} 
\label{lem:SplusSp}
For $i\in \mathbb Z^+$,
\begin{eqnarray*}
S_i + S'_i = \frac{
 ({\rm bad}\,A)_{i} \, {\rm ad}\,A}{q^i-q^{-i}}.
 \end{eqnarray*}
 \end{lemma}
\begin{proof} Use Definition
\ref{def:Tn} and
Lemma
\ref{lem:plus}.
\end{proof}

\begin{lemma} 
\label{lem:adad}
For $i \in \mathbb Z^+$,
\begin{eqnarray*}
\frac{
 ({\rm ad}_{i} A)
 ({\rm ad}_{-i} A)
 }
 {(q^{2i}-q^{-2i})^2}
 +
  I
 =
 \frac{q^{2i+1}-q^{-2i-1}}{q^{2i}-q^{-2i}}\,
 {\rm bad}_{i} A.
 \end{eqnarray*}
\end{lemma}
\begin{proof} Use
Definition
\ref{def:bad}.
\end{proof}

\begin{lemma} 
\label{lem:SS}
For $i \in \mathbb Z^+$,
\begin{eqnarray*}
S_i S'_i 
 + 
 ({\rm bad}\,A)^2_{i}
 = \frac{q^{2i+1}-q^{-2i-1}}{q^{2i}-q^{-2i}}\,
 ({\rm bad}\,A)_{i}
 ({\rm bad}\,A)_{i+1}.
 \end{eqnarray*}
 \end{lemma}
\begin{proof} Use Definition
\ref{def:Tn}
and
Lemma
\ref{lem:adad}.
\end{proof}

\begin{lemma} 
\label{lem:pmAd}
For $i,j \in \mathbb Z^+$,
\begin{eqnarray*}
&&
\frac{
 ({\rm ad}_{i} A)
 ({\rm ad}_{-j} A)+
 ({\rm ad}_{-i} A)
 ({\rm ad}_{j} A)}
 {(q^{2i}-q^{-2i})(q^{2j}-q^{-2j})}
 +
 (q^{i-j}+ q^{j-i}) I
\\
&&\qquad \qquad 
 =\quad  
 \frac{q^{2i+1}-q^{-2i-1}}{q^{i+j}-q^{-i-j}}\,
 {\rm bad}_{i} A
 +
 \frac{q^{2j+1}-q^{-2j-1}}{q^{i+j}-q^{-i-j}}\,
 {\rm bad}_{j} A.
 \end{eqnarray*}

\end{lemma}
\begin{proof} Routine using
Definition
\ref{def:bad}.
\end{proof}

\begin{lemma} 
\label{lem:pmSS}
For $i,j \in \mathbb Z^+$,
\begin{eqnarray*}
&&
S_i S'_j + S'_i S_j 
 + (q^{i-j}+ q^{j-i})
 ({\rm bad}\,A)_{i}
 ({\rm bad}\,A)_{j}
\\
&&
\qquad =\quad  
 \frac{q^{2i+1}-q^{-2i-1}}{q^{i+j}-q^{-i-j}}\,
 ({\rm bad}\,A)_{i+1}
 ({\rm bad}\,A)_{j}
 + \frac{q^{2j+1}-q^{-2j-1}}{q^{i+j}-q^{-i-j}}\,
 ({\rm bad}\,A)_{i}
 ({\rm bad}\,A)_{j+1}.
 \end{eqnarray*}
 \end{lemma}
\begin{proof} Use Definition
\ref{def:Tn}
and
Lemma
\ref{lem:pmAd}.
\end{proof}

%%\noindent The following three propositions are consequences of
%%the identities in the Appendix.
\begin{proposition}
\label{prop:TTp}
For $n \in \mathbb N$,
\begin{eqnarray*}
\Biggl(\sum_{i=0}^n S_i\Biggr)
\Biggl(\sum_{j=0}^n S'_j\Biggr) =
I+ ({\rm bad}\,A)_{n+1} 
\Biggl(
\sum_{r=0}^{n-1} 
\frac{q^{2n+1}-q^{-2n-1}}{q^{n+r+1}-q^{-n-r-1}}
({\rm bad}\,A)_{r+1} 
\Biggr).
\end{eqnarray*}
\end{proposition}
\begin{proof} 
The proof is by induction on $n$. 
Let $D_n$ denote the left-hand side minus the right-hand side.
We show that $D_n=0$. 
 One routinely obtains
$D_0=0$, so assume $n\geq 1$.
To show that $D_n=0$, it suffices to show
 that $D_n - D_{n-1}=0$. 
In the expression $D_n - D_{n-1}$, eliminate
the terms $\lbrace S_i S'_n\rbrace_{i=0}^{n-1}$, 
$\lbrace S_n S'_j\rbrace_{j=0}^{n-1}$, $S_n S'_n$
using Lemmas
\ref{lem:SplusSp},
\ref{lem:SS},
\ref{lem:pmSS}. After a routine simplification we
obtain 
 $D_n - D_{n-1}=0$, so
 $D_n=0$.
\end{proof}

%\noindent For each of the following two lemmas the proof is
%routine and omitted.

\noindent Our next goal is to prove Proposition
\ref{prop:TXY}
below. To this end we give some more identities
that hold in $\mathcal A$.

\begin{lemma}
\label{lem:1a}
For distinct  $i,j \in \mathbb Z$ and $X,Y \in \mathcal A$,
\begin{eqnarray*}
XA = \frac{q^{j} {\rm ad}_i A - q^i {\rm ad}_{j} A}{q^{i-j}-q^{j-i}} \,(X),
\qquad \qquad 
AY = \frac{q^{-j} {\rm ad}_i A - q^{-i} {\rm ad}_j A}{q^{i-j}-q^{j-i}}
\,(Y).
\end{eqnarray*}
\end{lemma}
\begin{proof} Routine.
\end{proof}

\begin{lemma}
\label{lem:adbad}
For  $i \in \mathbb Z$ and $j \in \mathbb Z^+$ and
$X, Y \in \mathcal A$,
\begin{eqnarray*}
&&
{\rm ad}_i A\,
 ({\rm bad}\,A)_{j} (X) =
q^{i-j} (q^{2j}-q^{-2j})S_{j}(X)
+
q^{-j}(q^{i-j}-q^{j-i})
\Bigl(
({\rm bad}\,A)_{j}(X)
\Bigr) A,
\\
&&
{\rm ad}_i A\,
 ({\rm bad}\,A)_{j} (Y) =
q^{j-i} (q^{2j}-q^{-2j})S_j(Y)
+
q^{j}(q^{i-j}-q^{j-i})
A \Bigl(
({\rm bad}\,A)_{j}(Y)
\Bigr).
\end{eqnarray*}
\end{lemma}
\begin{proof} Use
Definition \ref{def:Tn}
and 
Lemma 
\ref{lem:1a}.
\end{proof}

\begin{lemma}
\label{lem:adiSj}
For $i\in \mathbb Z$ and
 $j\in \mathbb N$ and
$X, Y \in \mathcal A$,
\begin{eqnarray*}
&&
{\rm ad}_i A\,
 S_{j} (X) =
q^{i+j} (q^{2j+1}-q^{-2j-1})
({\rm bad}\,A)_{j+1}(X)
-
q^{i+j}(q^{2j}-q^{-2j})
({\rm bad}\,A)_{j}(X)
\\
&& 
\qquad \qquad \qquad +
\;
q^{j}(q^{i+j}-q^{-i-j})
\Bigl(S_j(X)\Bigr) A,
\\
&&
{\rm ad}_i A\,
 S_{j} (Y) =
q^{-i-j} (q^{2j+1}-q^{-2j-1})
({\rm bad}\,A)_{j+1}(Y)
-
q^{-i-j}(q^{2j}-q^{-2j})
({\rm bad}\,A)_{j}(Y)
\\
&&
\qquad \qquad 
\qquad +\;
q^{-j}(q^{i+j}-q^{-i-j})
A \Bigl(S_j(Y)\Bigr).
\end{eqnarray*}
\end{lemma}
\begin{proof} Use
Definitions
\ref{def:bad},
\ref{def:Tn}
and 
Lemma
\ref{lem:1a}.
\end{proof}

\begin{lemma}
\label{lem:1}
For $h,i,j \in \mathbb Z$ and  $X,Y \in \mathcal A$,
\begin{eqnarray*}
&&
 {\rm ad}_{h} A \, (XY)
= q^{h-i} \bigl({\rm ad}_{i} A \, (X)\bigr) Y
+ q^{j-h} X
\bigl(
{\rm ad}_{j} A \, (Y)\bigr) + 
q^{j-i}(q^{h-i-j}-q^{i+j-h})XAY.
\end{eqnarray*}
\end{lemma}
\begin{proof} Routine.
\end{proof}

%%%%%%%%%%%%%%%%%%%%%%%%%%%%%%%%%%%%%%%%%%%%%%%%%%%%%%%
%\begin{lemma}
%\label{lem:1}
%For $i,j \in \mathbb Z$ and  $X,Y \in \mathcal A$,
%\begin{eqnarray*}
%&&
% {\rm ad}_{i+j} A \, (XY)
%= q^{j} \bigl({\rm ad}_{i} A \, (X)\bigr) Y
%+ q^{-i} X
%\bigl(
%{\rm ad}_{j} A \, (Y)\bigr).
%\end{eqnarray*}
%\end{lemma}
%\begin{proof} Routine.
%\end{proof}
%%%%%%%%%%%%%%%%%%%%%%%%%%%%%%%%%%%%%%%%%%%%%%%%%%%%

\noindent The next four lemmas are routinely obtained using
Lemmas
\ref{lem:1a}--\ref{lem:1}.

\begin{lemma}
\label{lem:adaBB}
For  $h\in \mathbb Z$ and $i,j \in \mathbb Z^+$ and
$X, Y \in \mathcal A$,
\begin{eqnarray*}
&&
{\rm ad}_{h} A \, 
\biggl(
\Bigl(
({\rm bad}\,A)_i (X)
\Bigr)
\Bigl(
({\rm bad}\,A)_j (Y)
\Bigr)
\biggr)
=
q^{h-i} (q^{2i}-q^{-2i})
\Bigl(S_i(X)\Bigr)
\Bigl(
({\rm bad}\,A)_j (Y)
\Bigr)
\\
&& \qquad \qquad \qquad \qquad 
+
\;q^{j-h}(q^{2j}-q^{-2j})
\Bigl(
({\rm bad}\,A)_i (X)
\Bigr)
\Bigl(S_j(Y)\Bigr)
\\
&& \qquad \qquad \qquad \qquad 
+
\;q^{j-i}(q^{h-i-j}-q^{i+j-h})
\Bigl(
({\rm bad}\,A)_i (X)
\Bigr) A
\Bigl(
({\rm bad}\,A)_j (Y)
\Bigr).
\end{eqnarray*}
\end{lemma}

\begin{lemma}
\label{lem:adaSS}
For  $h\in \mathbb Z$ and $i,j \in \mathbb N$ and
$X, Y \in \mathcal A$,
\begin{eqnarray*}
&&
{\rm ad}_{h} A \, 
\biggl(
\Bigl(
S_i (X)
\Bigr)
\Bigl(
S_j (Y)
\Bigr)
\biggr)
=
q^{h+i} (q^{2i+1}-q^{-2i-1})
\Bigl(
({\rm bad}\,A)_{i+1} (X)
\Bigr)
\Bigl(
S_j(Y)
\Bigr)
\\
&& \qquad \qquad \qquad \qquad 
-
\;q^{h+i}(q^{2i}-q^{-2i})
\Bigl(
({\rm bad}\,A)_i (X)
\Bigr)
\Bigl(S_j(Y)\Bigr)
\\
&& \qquad \qquad \qquad \qquad 
+
\;q^{-h-j}(q^{2j+1}-q^{-2j-1})
\Bigl(
S_i (X)
\Bigr)
\Bigl(
({\rm bad}\,A)_{j+1} (Y)
\Bigr)
\\
&& \qquad \qquad \qquad \qquad 
-
\;q^{-h-j}(q^{2j}-q^{-2j})
\Bigl(
S_i (X)
\Bigr)
\Bigl(
({\rm bad}\,A)_{j} (Y)
\Bigr)
\\
&& \qquad \qquad \qquad \qquad 
+
\;q^{i-j}(q^{h+i+j}-q^{-h-i-j})
\Bigl(
S_i (X)
\Bigr)A
\Bigl(
S_{j} (Y)
\Bigr).
\end{eqnarray*}
\end{lemma}

\begin{lemma}
\label{lem:adaSB}
For  $h\in \mathbb Z$ and $i \in \mathbb N$ and 
$j \in \mathbb Z^+$ and
$X, Y \in \mathcal A$,
\begin{eqnarray*}
&&
{\rm ad}_{h} A \, 
\biggl(
\Bigl(
S_i (X)
\Bigr)
\Bigl(
({\rm bad}\,A)_{j} (Y)
\Bigr)
\biggr)
=
q^{h+i} (q^{2i+1}-q^{-2i-1})
\Bigl(
({\rm bad}\,A)_{i+1} (X)
\Bigr)
\Bigl(
({\rm bad}\,A)_{j} (Y)
\Bigr)
\\
&& \qquad \qquad \qquad \qquad 
-
\;q^{h+i}(q^{2i}-q^{-2i})
\Bigl(
({\rm bad}\,A)_i (X)
\Bigr)
\Bigl(
({\rm bad}\,A)_{j} (Y)
\Bigr)
\\
&& \qquad \qquad \qquad \qquad 
+
\;q^{j-h}(q^{2j}-q^{-2j})
\Bigl(
S_i (X)
\Bigr)
\Bigl(
S_{j} (Y)
\Bigr)
\\
&& \qquad \qquad \qquad \qquad 
+
\;q^{i+j}(q^{h+i-j}-q^{j-h-i})
\Bigl(
S_i (X)
\Bigr)A
\Bigl(
({\rm bad}\,A)_{j} (Y)
\Bigr).
\end{eqnarray*}
\end{lemma}

\begin{lemma}
\label{lem:adaBS}
For  $h\in \mathbb Z$ and $i \in \mathbb Z^+$ and 
$j \in \mathbb N$ and
$X, Y \in \mathcal A$,
\begin{eqnarray*}
&&
{\rm ad}_{h} A \, 
\biggl(
\Bigl(
({\rm bad}\,A)_{i} (X)
\Bigr)
\Bigl(
S_j (Y)
\Bigr)
\biggr)
=
q^{-h-j} (q^{2j+1}-q^{-2j-1})
\Bigl(
({\rm bad}\,A)_{i} (X)
\Bigr)
\Bigl(
({\rm bad}\,A)_{j+1} (Y)
\Bigr)
\\
&& \qquad \qquad \qquad \qquad 
-
\;q^{-h-j}(q^{2j}-q^{-2j})
\Bigl(
({\rm bad}\,A)_i (X)
\Bigr)
\Bigl(
({\rm bad}\,A)_{j} (Y)
\Bigr)
\\
&& \qquad \qquad \qquad \qquad 
+
\;q^{h-i}(q^{2i}-q^{-2i})
\Bigl(
S_i (X)
\Bigr)
\Bigl(
S_{j} (Y)
\Bigr)
\\
&& \qquad \qquad \qquad \qquad 
+
\;q^{-i-j}(q^{h-i+j}-q^{i-h-j})
\Bigl(
({\rm bad}\,A)_{i} (X)
\Bigr)A
\Bigl(
S_j (Y)
\Bigr).
\end{eqnarray*}
\end{lemma}

\begin{proposition}
\label{prop:TXY}
For $n \in \mathbb N$  and $X,Y \in \mathcal A$,
\begin{eqnarray*}
&&
\sum_{i=0}^n S_i(XY)= 
\sum_{\stackrel{r,s \in \mathbb N}{
r+s\leq n}} S_r(X) S_s(Y)
%%%%\prod_{\stackrel{0 \leq j \leq d}{j \neq i}}
+ 
\sum_{\stackrel{r,s\in \mathbb N}{r+s=n-1}} 
\Bigl(
({\rm bad}\,A)_{r+1} (X)
\Bigr)
\Bigl(
({\rm bad}\,A)_{s+1} (Y)
\Bigr) q^{r-s},
\\
&&({\rm bad}\,A)_{n+1}(XY)
= \sum_{\stackrel{r,s\in \mathbb N}{r+s=n}} S_r(X) 
\Bigl(
({\rm bad}\,A)_{s+1}(Y)
\Bigr)
q^{-r}
+
\sum_{\stackrel{r,s\in \mathbb N}{r+s=n}}
\Bigl(
({\rm bad}\,A)_{r+1} (X)
\Bigr)
S_s(Y) q^{s}
\\
&& \qquad \qquad \qquad \qquad  -\;
\sum_{\stackrel{r,s\in \mathbb N}{r+s=n-1}} 
\Bigl(
({\rm bad}\,A)_{r+1} (X)
\Bigr)
A 
\Bigl(
({\rm bad}\,A)_{s+1} (Y)
\Bigr).
\end{eqnarray*}
\end{proposition}
\begin{proof} The proof is by induction on $n$.
Let $S(n)$ (resp. $B(n)$) denote the first
(resp.
second) displayed equation in the proposition
statement. The equation $S(0)$ holds since $S_0$ the identity map.
The equation $B(0)$ holds by
Lemma
\ref{lem:1} at $h=0$, $i=0$, $j=0$.
To get $S(n+1)$ from $S(n)$ and $B(n)$, 
apply ${\rm ad}_{n+1} A$ to each side of
$B(n)$, and evaluate the result using
Lemmas
\ref{lem:adaBB}--\ref{lem:adaBS}.
One obtains $S(n+1)-S(n)$ after a
brief calculation.
For the rest of this proof, assume that  $n\geq 1$.
To get $B(n)$ from $S(n)$ and $S(n-1)$,
apply ${\rm ad}_{-n} A$ to each side of
$S(n)-S(n-1)$, and evaluate the result using
Lemmas
\ref{lem:adaBB}--\ref{lem:adaBS}.
 One obtains
$B(n)$ 
after a brief calculation.
The result follows from these comments.
\end{proof}

\noindent The following result is a variation on
Proposition
\ref{prop:TXY}.

\begin{proposition}
\label{prop:primever}
For $n \in \mathbb N$ and $X,Y \in \mathcal A$,
\begin{eqnarray*}
&&
\sum_{i=0}^n S'_i(XY)= \sum_{\stackrel{r,s\in \mathbb N}{r+s\leq n}}
S'_r(X) S'_s(Y)
+ 
\sum_{\stackrel{r,s\in \mathbb N}{r+s=n-1}} 
\Bigl(
({\rm bad}\,A)_{r+1} (X)
\Bigr)
\Bigl(
({\rm bad}\,A)_{s+1} (Y)
\Bigr) q^{s-r},
\\
&&({\rm bad}\,A)_{n+1}(XY)
= \sum_{\stackrel{r,s\in \mathbb N}{r+s=n}} S'_r(X) 
\Bigl(
({\rm bad}\,A)_{s+1}(Y)
\Bigr)
q^{r}
+
\sum_{\stackrel{r,s\in \mathbb N}{r+s=n}}
\Bigl(
({\rm bad}\,A)_{r+1} (X)
\Bigr)
S'_s(Y) q^{-s}
\\
&& \qquad \qquad \qquad \qquad  +\;
\sum_{\stackrel{r,s\in \mathbb N}{r+s=n-1}} 
\Bigl(
({\rm bad}\,A)_{r+1} (X)
\Bigr)
A 
\Bigl(
({\rm bad}\,A)_{s+1} (Y)
\Bigr).
\end{eqnarray*}
\end{proposition}
\begin{proof}  In Proposition
 \ref{prop:TXY}, replace $q$ by $q^{-1}$ and
 evaluate the result using
 Definitions
 \ref{def:bad}, \ref{def:Tn}.
\end{proof}

\noindent Strictly speaking we do not need the
following result; we mention it for the sake
of completeness.

\begin{proposition}
\label{prop:BAD}
For $n \in \mathbb N$ and $X,Y \in \mathcal A$,
\begin{eqnarray*}
&&({\rm bad}\,A)_{n+1}(XY)\\
&& \qquad =\quad 
\sum_{\stackrel{r,s\in \mathbb N}{r+s=n}} S_r(X) 
\Bigl(
({\rm bad}\,A)_{s+1}(Y)
\Bigr)
q^{-r}
+
\sum_{\stackrel{r,s\in \mathbb N}{r+s=n}}
\Bigl(
({\rm bad}\,A)_{r+1} (X)
\Bigr)
S'_s(Y) q^{-s}
\\
&& \qquad =\quad 
\sum_{\stackrel{r,s\in \mathbb N}{r+s=n}} S'_r(X) 
\Bigl(
({\rm bad}\,A)_{s+1}(Y)
\Bigr)
q^{r}
+
\sum_{\stackrel{r,s\in \mathbb N}{r+s=n}}
\Bigl(
({\rm bad}\,A)_{r+1} (X)
\Bigr)
S_s(Y) q^{s}.
\end{eqnarray*}
\end{proposition}
\begin{proof}  
Using Definition \ref{def:Tn}
and 
Lemma 
\ref{lem:1a} we find 
that for $i\in \mathbb Z^+$,
\begin{eqnarray}
&&
q^{-i}S_i(X)-q^i S'_i(X) = 
\Bigl(
({\rm bad}\,A)_{i} (X)
\Bigr)A,
\label{eq:Sp1}
\\
&&
q^{i}S_i(Y)-q^{-i} S'_i(Y) = 
A \Bigl(
({\rm bad}\,A)_{i} (Y)
\Bigr).
\label{eq:Sp2}
\end{eqnarray}
In the relations from the proposition statement,
eliminate the terms $S'_r(X)$, $S'_s(Y)$ using
(\ref{eq:Sp1}),
(\ref{eq:Sp2})
and compare the
results with the second equation in
Proposition
\ref{prop:TXY}.
\end{proof}

\begin{proposition}
\label{lem:rs}
Given  $r,s \in \mathbb N$ and $X,Y \in \mathcal A$
such that
\begin{eqnarray*}
({\rm bad}\, A)_{r+1} (X)=0, 
\qquad \qquad 
({\rm bad}\, A)_{s+1} (Y)=0.
\end{eqnarray*}
Then 
\begin{eqnarray*}
({\rm bad}\, A)_{r+s+1} (XY)=0.
\end{eqnarray*}
\end{proposition}
\begin{proof} Use the second equation in
Proposition
\ref{prop:TXY} or
\ref{prop:primever}. Alternatively
use either equation in
Proposition
\ref{prop:BAD}.
\end{proof}

\section{The subalgebra $\mathcal A^\vee$}

\noindent We continue to work with the element $A$ of
the $\mathbb F$-algebra $\mathcal A$.

\begin{definition} 
\label{def:An}
\rm For $n \in \mathbb N$ let ${\mathcal A}^{(n)}$ denote
the set of elements in $\mathcal A$ at which $({\rm bad}\,A)_{n+1}$
vanishes. Note that 
 ${\mathcal A}^{(n)}$ is a subspace of the $\mathbb F$-vector space
 $\mathcal A$.
\end{definition}

\begin{example}
\label{ex:com}
\rm 
The subspace $\mathcal A^{(0)}$ consists of the elements in
$\mathcal A$ that commute with $A$.
\end{example}

\begin{example}
\label{ex:ex1}
\rm 
The subspace $\mathcal A^{(1)}$ consists of the elements $X$ in
$\mathcal A$ such that
\begin{eqnarray*}
A^3X-\lbrack 3\rbrack_q A^2XA+ 
\lbrack 3\rbrack_q AXA^2 -XA^3 = (q^2-q^{-2})^2 (XA-AX).
\end{eqnarray*}
\end{example}

\begin{lemma} 
\label{lem:incl}
We have
$\mathcal A^{(n)} \subseteq 
\mathcal A^{(n+1)}$ for $n \in \mathbb N$.
\end{lemma}
\begin{proof} 
By (\ref{eq:bad}) and
Definition
\ref{def:An}.
\end{proof}

\begin{lemma}
\label{lem:Tdef}
Pick $n \in \mathbb N$. Then for $r>n$ the maps $S_r, S'_r$ vanish
on ${\mathcal A}^{(n)}$. Moreover on 
${\mathcal A}^{(n)}$,
\begin{eqnarray}
S = \sum_{r=0}^n S_r, \qquad \qquad 
S' = \sum_{r=0}^n S'_r.
\label{eq:Tshortsum}
\end{eqnarray}
\end{lemma}
\begin{proof} 
By 
(\ref{eq:bad})
the map
 $({\rm bad}\,A)_{n+1} $ is a factor of
$({\rm bad}\,A)_r$. 
By Definition
\ref{def:Tn},
 the map $({\rm bad}\,A)_r$ is a factor of
$S_r$ and $S'_r$.
Consequently 
 $S_r$ and $S'_r$ vanish
on ${\mathcal A}^{(n)}$.
The equations
(\ref{eq:Tshortsum}) are from
Lemma
\ref{def:fs}.
\end{proof}

\noindent By Lemma
\ref{lem:Tdef},
$S$ and $S'$ are well defined $\mathbb F$-linear maps on 
$\mathcal A^{(n)}$ for all $n \in \mathbb N$.

\begin{lemma}
\label{lem:Ainv}
For $n \in \mathbb N$ the
subspace $\mathcal A^{(n)}$ is invariant under $S$ and $S'$.
\end{lemma}
\begin{proof} The map $({\rm bad}\,A)_{n+1}$ commutes with
 $S_r$ and $S'_r$ for $r \in \mathbb N$.
\end{proof}

\begin{lemma}
\label{lem:wdi}
For $n \in \mathbb N$ the maps
$S: \mathcal A^{(n)}\to \mathcal A^{(n)}$
and
$S': \mathcal A^{(n)}\to \mathcal A^{(n)}$
are inverses.
\end{lemma}
\begin{proof} By Proposition
\ref{prop:TTp}.
\end{proof}

\begin{example}
\label{eq:ex0}
The maps $S$ and  $S'$ fix everything in 
$\mathcal A^{(0)}$.
\end{example}
\begin{proof} On ${\mathcal A}^{(0)}$ we have 
$S=S_0=I$
and $S'=S'_0=I$.
\end{proof}

\begin{example}
\label{eq:ex1}
Pick $X \in \mathcal A^{(1)}$.
Then $S$ sends
\begin{eqnarray*}
X \mapsto 
 X + \frac{qA^2X-(q+q^{-1})AXA+q^{-1}XA^2}{(q-q^{-1})(q^2-q^{-2})},
 \end{eqnarray*}
 and $S'$ sends
\begin{eqnarray*}
X \mapsto 
 X + \frac{q^{-1}A^2X-(q+q^{-1})AXA+qXA^2}{(q-q^{-1})(q^2-q^{-2})}.
 \end{eqnarray*}
\end{example}
\begin{proof} On ${\mathcal A}^{(1)}$ we have
$S=S_0+S_1$ and
$S'=S'_0+S'_1$.
\end{proof}

\begin{lemma}
\label{lem:mult}
We have ${\mathcal A}^{(r)} {\mathcal A}^{(s)} \subseteq
{\mathcal A}^{(r+s)}$
for $r,s\in \mathbb N$.
\end{lemma}
\begin{proof}
By Proposition
\ref{lem:rs}.
\end{proof}

\begin{definition}
\label{def:Av}
\rm Define
$\mathcal A^\vee = 
\cup_{n \in \mathbb N} {\mathcal A}^{(n)}$.
\end{definition}

\begin{lemma} 
\label{lem:subalgA}
The set
${\mathcal A}^\vee$ is a
subalgebra of $\mathcal A$ that contains $A$.
\end{lemma} 
\begin{proof} By 
Definition
\ref{def:An}
and
Lemma 
\ref{lem:incl},
$\mathcal A^\vee$ is a
subspace of the $\mathbb F$-vector space $\mathcal A$.
By
Example
\ref{ex:com},
the subspace $\mathcal A^\vee$ contains 1 and $A$.
By Lemma
\ref{lem:mult},
the subspace 
$\mathcal A^\vee$
is closed under multiplication.
The result follows.
\end{proof}

\noindent By Definition
\ref{def:Av} along with
Lemma
\ref{lem:Ainv}
and the comment above it,
we obtain $\mathbb F$-linear maps $S:
{\mathcal A}^\vee \to
{\mathcal A}^\vee$
and
$S':
{\mathcal A}^\vee \to
{\mathcal A}^\vee$.

\begin{proposition}
\label{prop:TTpinv}
The maps $S$ and $S'$ act on
the algebra ${\mathcal A}^\vee$ as an automorphism, and
these automorphisms
 are inverses.
\end{proposition}
\begin{proof} To get the first assertion use
Propositions
\ref{prop:TXY}, \ref{prop:primever}.
The last assertion follows from
Lemma
\ref{lem:wdi}.
\end{proof}

\section{$A$-Standard algebras and their Lusztig automorphism}

We continue to work with the element $A$ of the
$\mathbb F$-algebra $\mathcal A$.

\begin{definition}
\label{def:st}
\rm An element $X \in \mathcal A$ is called 
{\it $A$-standard} 
whenever there  exists a positive integer
$n$ such that
$({\rm bad}\,A)_n (X)=0$. Note that ${\mathcal A}^\vee$ consists
of the $A$-standard elements of $\mathcal A$.
\end{definition}

\begin{definition} 
\label{def:st2}
\rm The algebra $\mathcal A$ is called
{\it $A$-standard} whenever each element of $\mathcal A$
is $A$-standard.
\end{definition}

\begin{lemma} 
\label{lem:three}
The following {\rm (i)--(iii)} are equivalent:
\begin{enumerate}
\item[\rm (i)] 
 $\mathcal A$ is $A$-standard;
\item[\rm (ii)] 
$\mathcal A^\vee=\mathcal A$;
\item[\rm (iii)] 
$\mathcal A$ has a generating set whose elements are $A$-standard.
\end{enumerate}
\end{lemma}
\begin{proof} ${{\rm (i)} \Leftrightarrow {\rm (ii)}}$
By Definitions \ref{def:st},
\ref{def:st2}.
\\
\noindent 
${{\rm (i)} \Rightarrow {\rm (iii)}}$ Clear.
\\
${{\rm (iii)} \Rightarrow {\rm (ii)}}$
By Lemma
\ref{lem:subalgA}
 $\mathcal A^\vee$ is a subalgebra of $\mathcal A$.
By Definition
\ref{def:st}
 $\mathcal A^\vee$ contains
each $A$-standard element of $\mathcal A$.
The result follows.
\end{proof}

\begin{theorem} 
\label{thm:standard}
Assume that $\mathcal A$ is $A$-standard.
Then $S$ and $S'$ act on
$\mathcal A$ as an automorphism, and these
 automorphisms
 are inverses.
\end{theorem}
\begin{proof} Apply Proposition
\ref{prop:TTpinv}
to the algebra $\mathcal A^\vee = \mathcal A$.
\end{proof}

\noindent Recall the $q$-Onsager algebra $\mathcal O_q$
and its generators $A$, $B$.
\begin{proposition} 
\label{prop:LS} For $\mathcal O_q$ the following
{\rm (i)--(iv)} hold:
\begin{enumerate}
\item[\rm (i)]  $A \in \mathcal O_q^{(0)}$ and
 $B \in \mathcal O_q^{(1)}$;
\item[\rm (ii)] 
the algebra $\mathcal O_q$
is $A$-standard;
\item[\rm (iii)] $S$ sends
\begin{eqnarray*}
&&A\mapsto A, \qquad \qquad
B \mapsto B + \frac{qA^2B-(q+q^{-1})ABA+q^{-1}BA^2}{(q-q^{-1})(q^2-q^{-2})}
\end{eqnarray*}
and 
$S'$ sends
\begin{eqnarray*}
A\mapsto A, \qquad \qquad
B\mapsto B + \frac{q^{-1}A^2B-(q+q^{-1})ABA+qBA^2}{(q-q^{-1})(q^2-q^{-2})};
\end{eqnarray*}
\item[\rm (iv)] 
 $S=L$ and $S'=L^{-1}$.
\end{enumerate}
\end{proposition}
\begin{proof} (i) 
We have $A \in 
 \mathcal O_q^{(0)}$ by
Example
\ref{ex:com},
and 
$B \in 
 \mathcal O_q^{(1)}$ by
Example
\ref{ex:ex1}.
\\
\noindent (ii) The generators 
$A$, $B$ are 
$A$-standard by (i) above.
Now $\mathcal O_q$
is $A$-standard by Lemma
\ref{lem:three}. 
\\
\noindent (iii) By (i) above and
%%Comparing (\ref{eq:bk1}),
%% (\ref{eq:bk2})
Examples 
\ref{eq:ex0},
\ref{eq:ex1}.
\\
\noindent (iv)
Compare (\ref{eq:bk1}),
 (\ref{eq:bk2})
with (iii) above.
\end{proof}

\noindent Theorem
\ref{conj:T}
follows from Theorem
\ref{thm:standard}
and
Proposition
\ref{prop:LS}.
Combining
 Theorem
\ref{thm:standard}
and
Proposition
\ref{prop:LS}(i)--(iii), we get a computer-free proof that
there exists an automorphism 
$L$ of
$\mathcal O_q$ that satisfies
 (\ref{eq:bk1}),
 (\ref{eq:bk2}).
\medskip

\noindent We return our attention to the algebra $\mathcal A$, and the element
$A \in \mathcal A$. The following definition is motivated by
Proposition 
\ref{prop:LS}.
\begin{definition}
\label{def:LA} \rm
Assume that  $\mathcal A$ is $A$-standard,
and consider its automorphism $S$ from Theorem
\ref{thm:standard}. We call $S$ the {\it Lusztig automorphism
of $\mathcal A$ that corresponds to $A$}.
\end{definition}

\section{The higher order $q$-Dolan/Grady relations}

\noindent In this section we establish the 
higher order $q$-Dolan/Grady relations conjectured by Baseilhac
and Vu \cite{BVu}.
Let $\mathcal A$ denote an $\mathbb F$-algebra
and fix $A \in \mathcal A$. 

\begin{theorem} Given $X \in \mathcal A$ such that
\begin{eqnarray*}
A^3X-\lbrack 3\rbrack_q A^2XA+ 
\lbrack 3\rbrack_q AXA^2 -XA^3 = (q^2-q^{-2})^2 (XA-AX).
\end{eqnarray*}
Then 
\begin{eqnarray*}
({\rm bad}\, A)_{r+1} (X^r)=0 \qquad \qquad r\geq 1.
\end{eqnarray*}
We are using the notation
(\ref{eq:bad}).
\end{theorem}
\begin{proof} By Example
\ref{ex:ex1} we have
$X \in \mathcal A^{(1)}$.
By Lemma
\ref{lem:mult} we have
$X^r \in \mathcal A^{(r)}$.
The result follows by Definition
\ref{def:An}.
\end{proof}

\section{The current algebra $\mathcal A_q$ and its Lusztig automorphism
}
\noindent 
In \cite{basnc} Baseilhac and K. Shigechi
introduce the current algebra $\mathcal A_q$ for
 $\mathcal O_q$, and they
discuss how $\mathcal A_q$
is related to $\mathcal O_q$. This relationship is discussed further in 
\cite{basBel}, where it is conjectured 
that
$\mathcal O_q$ is a homomorphic image of $\mathcal A_q$
\cite[Conjecture~2]{basBel}.
The algebra $\mathcal A_q$ is defined by generators and
relations
\cite[Definition~3.1]{basnc}.
The generators are 
denoted 
$\mathcal W_{-k}$, $\mathcal W_{k+1}$,  $\mathcal G_{k+1}$, 
 $\mathcal{\tilde G}_{k+1}$,
where $k \in \mathbb N$.
In \cite[Lemma~2.1]{basBel}, Baseilhac and S. Belliard display
some central elements 
$\lbrace \Delta_{k+1}\rbrace_{k \in \mathbb N}$ for 
$\mathcal A_q$.
In \cite[Corollary~3.1]{basBel}, it is shown that $\mathcal A_q$ is generated by
these central elements together with
 $A=\mathcal W_0$ 
and $B=\mathcal W_1$.
The elements $A,B$
are known to satisfy the 
$q$-Dolan/Grady relations
(\ref{eq:dg1}), (\ref{eq:dg2}) \cite[eqn.~(3.7)]{basBel}.
In this section we show that $\mathcal A_q$ is
$A$-standard, and describe how the corresponding Lusztig automorphism
acts on the elements mentioned above.
We now recall the definition of $\mathcal A_q$.
\begin{definition}\rm
\label{def:Aq} (See 
\cite[Definition~3.1]{basnc}.)
Let $\mathcal A_q$ denote the $\mathbb F$-algebra with
generators
 $\mathcal W_{-k}$, $\mathcal W_{k+1}$,  $\mathcal G_{k+1}$,
 $\mathcal {\tilde G}_{k+1}$ $(k \in \mathbb N)$
 and the following relations:
\begin{eqnarray}
&&
 \lbrack \mathcal W_0, \mathcal W_{k+1}\rbrack= 
\lbrack \mathcal W_{-k}, \mathcal W_{1}\rbrack=
({\mathcal{\tilde G}}_{k+1} - \mathcal G_{k+1})/(q+q^{-1}),
\label{eq:3p1}
\\
&&
\lbrack \mathcal W_0, \mathcal G_{k+1}\rbrack_q= 
\lbrack {\mathcal{\tilde G}}_{k+1}, \mathcal W_{0}\rbrack_q= 
\rho  \mathcal W_{-k-1}-\rho 
 \mathcal W_{k+1},
\label{eq:3p2}
\\
&&
\lbrack \mathcal G_{k+1}, \mathcal W_{1}\rbrack_q= 
\lbrack \mathcal W_{1}, {\mathcal {\tilde G}}_{k+1}\rbrack_q= 
\rho  \mathcal W_{k+2}-\rho 
 \mathcal W_{-k},
\label{eq:3p3}
\\
&&
\lbrack \mathcal W_{-k}, \mathcal W_{-\ell}\rbrack=0,  \qquad 
\lbrack \mathcal W_{k+1}, \mathcal W_{\ell+1}\rbrack= 0,
\label{eq:3p4}
\\
&&
\lbrack \mathcal W_{-k}, \mathcal W_{\ell+1}\rbrack+
\lbrack \mathcal W_{k+1}, \mathcal W_{-\ell}\rbrack= 0,
\label{eq:3p5}
\\
&&
\lbrack \mathcal W_{-k}, \mathcal G_{\ell+1}\rbrack+
\lbrack \mathcal G_{k+1}, \mathcal W_{-\ell}\rbrack= 0,
\label{eq:3p6}
\\
&&
\lbrack \mathcal W_{-k}, {\mathcal {\tilde G}}_{\ell+1}\rbrack+
\lbrack {\mathcal {\tilde G}}_{k+1}, \mathcal W_{-\ell}\rbrack= 0,
\label{eq:3p7}
\\
&&
\lbrack \mathcal W_{k+1}, \mathcal G_{\ell+1}\rbrack+
\lbrack \mathcal  G_{k+1}, \mathcal W_{\ell+1}\rbrack= 0,
\label{eq:3p8}
\\
&&
\lbrack \mathcal W_{k+1}, {\mathcal {\tilde G}}_{\ell+1}\rbrack+
\lbrack {\mathcal {\tilde G}}_{k+1}, \mathcal W_{\ell+1}\rbrack= 0,
\label{eq:3p9}
\\
&&
\lbrack \mathcal G_{k+1}, \mathcal G_{\ell+1}\rbrack=0,
\qquad 
\lbrack {\mathcal {\tilde G}}_{k+1}, {\mathcal {\tilde G}}_{\ell+1}\rbrack= 0,
\label{eq:3p10}
\\
&&
\lbrack {\mathcal {\tilde G}}_{k+1}, \mathcal G_{\ell+1}\rbrack+
\lbrack \mathcal G_{k+1}, {\mathcal {\tilde G}}_{\ell+1}\rbrack= 0.
\label{eq:3p11}
\end{eqnarray}
In the above equations $\ell \in \mathbb N$
and $\rho = -(q^2-q^{-2})^2$.
 We are using the notation
 $\lbrack X,Y\rbrack=XY-YX$ and $\lbrack X,Y\rbrack_q=
qXY-q^{-1}YX$.
\end{definition}
\noindent For the algebra $\mathcal A_q$,
consider the element $A=\mathcal W_0$ and the
corresponding subspaces 
$\mathcal A^{(0)}_q$,
$\mathcal A^{(1)}_q$ from Examples
\ref{ex:com},
\ref{ex:ex1}.
\begin{lemma}
\label{lem:Aq} For the algebra $\mathcal A_q$ the following {\rm (i)--(v)} hold
for $k \in \mathbb N$:
\\
{\rm (i)} $\mathcal W_{-k} \in \mathcal A_q^{(0)}$;
\;\; {\rm (ii)} $\mathcal W_{k+1} \in \mathcal A_q^{(1)}$;
\;\; {\rm (iii)} $\mathcal G_{k+1} \in \mathcal A_q^{(1)}$;
\;\; {\rm (iv)} $\mathcal{\tilde G}_{k+1} \in \mathcal A_q^{(1)}$;
\;\; {\rm (v)}  $\Delta_{k+1} \in \mathcal A_q^{(0)}$.
\end{lemma}
\begin{proof} (i) The elements $\mathcal W_{-k}, \mathcal W_0$ commute by
(\ref{eq:3p4}). The result follows in view of
Example
\ref{ex:com}.
\\
\noindent (ii) We show that
\begin{eqnarray}
\label{eq:aneed}
\lbrack \mathcal W_0, \lbrack \mathcal W_0, \lbrack \mathcal W_0,
\mathcal W_{k+1} \rbrack \; \rbrack_q \rbrack_{q^{-1}}
= \rho \lbrack \mathcal W_0, \mathcal W_{k+1}
\rbrack.
\end{eqnarray}
By 
(\ref{eq:3p1}),
\begin{eqnarray}
\label{eq:start}
\lbrack \mathcal W_0, \lbrack \mathcal W_0, \lbrack \mathcal W_0,
\mathcal W_{k+1} \rbrack \; \rbrack_q \rbrack_{q^{-1}}
=
\frac{
\lbrack \mathcal W_0, 
\lbrack \mathcal W_0,
\mathcal{\tilde G}_{k+1}-\mathcal G_{k+1} \rbrack_q \rbrack_{q^{-1}}
}{q+q^{-1}}.
\end{eqnarray}
Using linear algebra and (\ref{eq:3p2}),
\begin{align}
%%%%%%%%%%\begin{eqnarray}
\lbrack \mathcal W_0,
\lbrack \mathcal W_0, 
\mathcal{\tilde G}_{k+1} \rbrack_q \rbrack_{q^{-1}}
&= \lbrack \mathcal W_0,
\lbrack \mathcal W_0, 
\mathcal{\tilde G}_{k+1} \rbrack_{q^{-1}} \rbrack_{q}
\nonumber
\\
&=
-\lbrack \mathcal W_0,
\lbrack 
\mathcal{\tilde G}_{k+1}, \mathcal W_0 \rbrack_{q} \rbrack_{q}
\nonumber
\\
&=
-\lbrack \mathcal W_0,
\lbrack 
\mathcal W_0,
\mathcal G_{k+1} \rbrack_{q} \rbrack_{q}.
\label{eq:WWG}
\end{align}
%%%\end{eqnarray}
Using in order
(\ref{eq:start}),
(\ref{eq:WWG}),
(\ref{eq:3p2}),
(\ref{eq:3p4})
we obtain
\begin{align*}
\lbrack \mathcal W_0, \lbrack \mathcal W_0, \lbrack \mathcal W_0,
\mathcal W_{k+1} \rbrack \; \rbrack_q \rbrack_{q^{-1}}
&=
\frac{
\lbrack \mathcal W_0, 
\lbrack \mathcal W_0,
\mathcal{\tilde G}_{k+1}-\mathcal G_{k+1} \rbrack_q \rbrack_{q^{-1}}
}{q+q^{-1}}
\\
&=
-\,\frac{
\lbrack \mathcal W_0, 
\lbrack \mathcal W_0,
\mathcal G_{k+1} 
\rbrack_q \rbrack_q +
\lbrack \mathcal W_0, 
\lbrack \mathcal W_0,
\mathcal G_{k+1} 
\rbrack_q \rbrack_{q^{-1}}
}{q+q^{-1}}
\\
&=- 
\lbrack \mathcal W_0, 
\lbrack \mathcal W_0,
\mathcal G_{k+1} 
\rbrack_q \rbrack
\\
&=
\rho \lbrack \mathcal W_0, \mathcal W_{k+1}- \mathcal W_{-k-1}
\rbrack
\\
&=
\rho \lbrack \mathcal W_0, \mathcal W_{k+1}\rbrack.
\end{align*}
\noindent We have shown 
(\ref{eq:aneed}). The result follows in view of Example
\ref{ex:ex1}.
\\
\noindent (iii) We show that
\begin{eqnarray}
\label{eq:Gneed}
\lbrack \mathcal W_0, \lbrack \mathcal W_0, \lbrack \mathcal W_0,
\mathcal G_{k+1} \rbrack \; \rbrack_q \rbrack_{q^{-1}}
= \rho \lbrack \mathcal W_0, \mathcal G_{k+1}
\rbrack.
\end{eqnarray}
Using in order (\ref{eq:3p2}),  
 (\ref{eq:3p4}),  
(\ref{eq:3p1})  we obtain 
\begin{align*}
\lbrack \mathcal W_0, \lbrack \mathcal W_0, \lbrack \mathcal W_0,
\mathcal G_{k+1} \rbrack \; \rbrack_q \rbrack_{q^{-1}}
&=
\lbrack \mathcal W_0, \lbrack \mathcal W_0, \lbrack \mathcal W_0,
\mathcal G_{k+1} \rbrack_q \rbrack\; \rbrack_{q^{-1}}
\\
&=
\rho 
\lbrack \mathcal W_0, \lbrack \mathcal W_0, 
\mathcal W_{-k-1} - \mathcal W_{k+1} \rbrack \;\rbrack_{q^{-1}}
\\
&=
-\rho 
\lbrack \mathcal W_0, \lbrack \mathcal W_0, 
\mathcal W_{k+1} \rbrack \;\rbrack_{q^{-1}}
\\
&=
\rho 
\frac{
\lbrack \mathcal W_0, 
\mathcal {G}_{k+1} - \mathcal{\tilde G}_{k+1} \rbrack_{q^{-1}}
}{q+q^{-1}}
\\
&=
\rho 
\frac{
\lbrack 
\mathcal {\tilde G}_{k+1} - \mathcal{G}_{k+1}, \mathcal W_0 \rbrack_{q}
}{q+q^{-1}}.
\end{align*}
Now in (\ref{eq:Gneed}), the left-hand side minus the right-hand side
is equal to
\begin{align*}
&
\rho\,\frac{
\lbrack
\mathcal {\tilde G}_{k+1},\mathcal W_0\rbrack_{q} - 
\lbrack \mathcal{G}_{k+1}, \mathcal W_0 \rbrack_{q}
}{q+q^{-1}} \;-\;\rho \lbrack \mathcal W_0, \mathcal G_{k+1}\rbrack
\\
&
\qquad \qquad \qquad \qquad 
=\;
\rho\,\frac{
\lbrack 
\mathcal{\tilde G}_{k+1},\mathcal W_0\rbrack_{q} - 
\lbrack \mathcal W_0, \mathcal{G}_{k+1} \rbrack_{q}
}{q+q^{-1}},
\end{align*}
and this is zero by (\ref{eq:3p2}).
We have shown
(\ref{eq:Gneed}). The result follows in view of
Example \ref{ex:ex1}.
\\
\noindent (iv) 
We show that
\begin{eqnarray}
\label{eq:tGneed}
\lbrack \mathcal W_0, \lbrack \mathcal W_0, \lbrack \mathcal W_0,
\mathcal{\tilde G}_{k+1} \rbrack \; \rbrack_q \rbrack_{q^{-1}}
= \rho \lbrack \mathcal W_0, \mathcal{\tilde G}_{k+1}
\rbrack.
\end{eqnarray}
Using in order
(\ref{eq:3p2}),
(\ref{eq:3p4}),
(\ref{eq:3p1}) we obtain
\begin{align*}
\lbrack \mathcal W_0, \lbrack \mathcal W_0, \lbrack \mathcal W_0,
\mathcal{\tilde G}_{k+1} \rbrack \; \rbrack_q \rbrack_{q^{-1}}
&=
\lbrack \mathcal W_0, \lbrack \mathcal W_0, \lbrack \mathcal W_0,
\mathcal{\tilde G}_{k+1} \rbrack_{q^{-1}} \rbrack\; \rbrack_{q}
\\
&=
-\lbrack \mathcal W_0, \lbrack \mathcal W_0, \lbrack
\mathcal{\tilde G}_{k+1},
\mathcal W_0
\rbrack_{q} \rbrack\; \rbrack_{q}
\\
&=
\rho 
\lbrack \mathcal W_0, \lbrack \mathcal W_0, 
\mathcal W_{k+1} - \mathcal W_{-k-1} \rbrack \;\rbrack_{q}
\\
&=
\rho 
\lbrack \mathcal W_0, \lbrack \mathcal W_0, 
\mathcal W_{k+1} \rbrack \;\rbrack_{q}
\\
&=
\rho 
\frac{
\lbrack  
\mathcal W_0,
\mathcal {\tilde G}_{k+1}-\mathcal G_{k+1}\rbrack_{q}
}{q+q^{-1}}.
\end{align*}
Now in (\ref{eq:tGneed}), the left-hand side minus the right-hand side
is equal to
\begin{align*}
&
\rho\,\frac{
\lbrack
\mathcal W_0,
\mathcal {\tilde G}_{k+1}\rbrack_{q} - 
\lbrack \mathcal W_0, \mathcal{G}_{k+1} \rbrack_{q}
}{q+q^{-1}} \;-\;\rho \lbrack \mathcal W_0, \mathcal{\tilde G}_{k+1}\rbrack
\\
&
\qquad \qquad \qquad \qquad 
=\;
\rho\,\frac{
\lbrack \mathcal{\tilde G}_{k+1}, \mathcal W_0 \rbrack_{q}
-
\lbrack \mathcal W_0, 
\mathcal {G}_{k+1}\rbrack_{q} 
}{q+q^{-1}},
\end{align*}
and this is zero by (\ref{eq:3p2}).
We have shown
(\ref{eq:tGneed}). The result follows in view of
Example \ref{ex:ex1}.
\\
\noindent (v) The element $\Delta_{k+1}$ is central,
so it commutes with $\mathcal W_0$. The result follows in view
of Example \ref{ex:com}.
\end{proof}

\begin{proposition}
\label{prop:Aqstandard}
The algebra $\mathcal A_q$ is $A$-standard.
\end{proposition}
\begin{proof} 
Consider the generators of $\mathcal A_q$ from
Definition
\ref{def:Aq}.
By Lemma \ref{lem:Aq}  these generators are
$A$-standard. 
Now $\mathcal A_q$ is $A$-standard 
by Lemma
\ref{lem:three}.
\end{proof}
\noindent Since 
the algebra $\mathcal A_q$ is $A$-standard, we may speak of
the corresponding Lusztig automorphism $S$ of $\mathcal A_q$,
from
Theorem
\ref{thm:standard} and Definition
\ref{def:LA}.

\begin{proposition} 
\label{conj:W}
For $k \in \mathbb N$ the automorphism $S$ sends
\begin{eqnarray*}
\mathcal W_{-k} &\mapsto& \mathcal W_{-k},
\\
\mathcal W_{k+1} &\mapsto& \mathcal W_{k+1} + \frac{q \mathcal W_0^2 \mathcal W_{k+1} -(q+q^{-1})\mathcal W_0 \mathcal W_{k+1}\mathcal W_0+ q^{-1}\mathcal W_{k+1}\mathcal W_0^2}{(q-q^{-1})(q^2-q^{-2})},
\\
\mathcal G_{k+1} &\mapsto& \mathcal G_{k+1} + \frac{q \mathcal W_0^2 
\mathcal G_{k+1} -(q+q^{-1})\mathcal W_0 \mathcal G_{k+1}\mathcal W_0+ q^{-1}\mathcal G_{k+1}
\mathcal W_0^2}{(q-q^{-1})(q^2-q^{-2})} = \mathcal{\tilde G}_{k+1},
\\
\mathcal {\tilde G}_{k+1} &\mapsto& \mathcal {\tilde G}_{k+1} + \frac{q \mathcal  W_0^2 \mathcal {\tilde G}_{k+1} -
(q+q^{-1})\mathcal W_0\mathcal {\tilde G}_{k+1}\mathcal W_0+ q^{-1}\mathcal {\tilde G}_{k+1}\mathcal W_0^2}{(q-q^{-1})(q^2-q^{-2})},
\\
\Delta_{k+1} &\mapsto& \Delta_{k+1}.
\end{eqnarray*} 
Moreover $S^{-1}$ sends
\begin{eqnarray*}
\mathcal W_{-k} &\mapsto& \mathcal W_{-k},
\\
\mathcal W_{k+1} &\mapsto& \mathcal W_{k+1} + \frac{q^{-1} \mathcal W_0^2\mathcal W_{k+1} -(q+q^{-1})\mathcal W_0 \mathcal W_{k+1}\mathcal W_0+ q \mathcal 
W_{k+1}\mathcal W_0^2}{(q-q^{-1})(q^2-q^{-2})},
\\
\mathcal G_{k+1} &\mapsto& \mathcal G_{k+1} + \frac{q^{-1} \mathcal W_0^2
\mathcal G_{k+1} -(q+q^{-1})\mathcal W_0\mathcal G_{k+1}\mathcal W_0+
q\mathcal G_{k+1}\mathcal W_0^2}{(q-q^{-1})(q^2-q^{-2})},
\\
\mathcal {\tilde G}_{k+1} &\mapsto& \mathcal{\tilde G}_{k+1} + 
\frac{q^{-1} \mathcal W_0^2 \mathcal {\tilde G}_{k+1} -(q+q^{-1})
\mathcal W_0\mathcal{\tilde G}_{k+1}\mathcal W_0+ q\mathcal {\tilde G}_{k+1}
\mathcal W_0^2}{(q-q^{-1})(q^2-q^{-2})}
= \mathcal G_{k+1},
\\
\Delta_{k+1} &\mapsto& \Delta_{k+1}.
\end{eqnarray*} 
\end{proposition}
\begin{proof} By
Theorem
\ref{thm:standard},
Lemma
\ref{lem:Aq}, 
and Examples
\ref{eq:ex0}, 
\ref{eq:ex1}.
\end{proof}

\section{Finite-dimensional $\mathcal O_q$-modules}

\noindent 
Recall the generators $A$, $B$ for the $q$-Onsager algebr $\mathcal O_q$.
Throughout this section $V$ denotes a finite-dimensional irreducible
 ${\mathcal O}_q$-module
 on which $A$ and $B$ are  
 diagonalizable. To avoid trivialities, we always assume that
 $V$ has 
 dimension at least 2.
We describe what happens when
$V$ is twisted via the Lusztig automorphism $L$ of $\mathcal O_q$.
By \cite[Theorem~3.10]{qSerre}
the elements $A$, $B$ act on $V$ as
 a tridiagonal pair. The tridiagonal pair concept is defined in
\cite[Definition~1.1]{TD00},
 and described further in
\cite{INT,
ItoTer,
IT:aug,
madrid}.
In what follows, we freely invoke the notation and
theory of tridiagonal pairs.
Fix a standard ordering $\lbrace \theta_i\rbrace_{i=0}^d$
of the eigenvalues of $A$ on $V$, and a
standard ordering $\lbrace \theta^*_i\rbrace_{i=0}^d$
of the eigenvalues of $B$ on $V$.
By construction 
 $\lbrace \theta_i\rbrace_{i=0}^d$ are mutually distinct and
 contained in $\mathbb F$. Similarly
 $\lbrace \theta^*_i\rbrace_{i=0}^d$ are mutually distinct and
 contained in $\mathbb F$. 
Note that $d\geq 1$; 
otherwise $A=\theta_0 I$ and $B=\theta^*_0I$ which contradicts the
irreducibility of $V$.
For $0 \leq i \leq d$ let 
$E_i:V\to V $
(resp. $E^*_i:V\to V $)
denote the projection onto the eigenspace of $A$ (resp. $B$)
for $\theta_i$ (resp. $\theta^*_i$).
By linear algebra,
\begin{eqnarray*}
A = \sum_{i=0}^d \theta_i E_i,
\qquad \qquad
B = \sum_{i=0}^d \theta^*_i E^*_i.
\end{eqnarray*}
By \cite[Lemma~2.4]{TD00}
the following hold for $0 \leq i,j\leq d$:
\begin{eqnarray}
E^*_i A E^*_j = 
\begin{cases}
0 &  {\mbox{\rm if $|i-j|>1$}}; \\
\not=0 & {\mbox{\rm if $|i-j|=1$}}
\end{cases}
\qquad \qquad
E_i B E_j = 
\begin{cases}
0 &  {\mbox{\rm if $|i-j|>1$}}; \\
\not=0 & {\mbox{\rm if $|i-j|=1$}}.
\end{cases}
\label{eq:tripProd}
\end{eqnarray}
\noindent
The following result can be found in \cite[Theorem~11.2]{TD00}; we give
a short proof for the sake of completeness.
\begin{lemma} {\rm (See \cite[Theorem~11.2]{TD00}.)}
There exist nonzero $a,b\in \mathbb F$
such that
\begin{eqnarray}
\theta_i = a q^{d-2i} + a^{-1}q^{2i-d},
\qquad \qquad
\theta^*_i = b q^{d-2i} + b^{-1}q^{2i-d}
\label{eq:ab}
\end{eqnarray}
for $0 \leq i \leq d$.
\end{lemma}
\begin{proof} We verify the equation on the left in
(\ref{eq:ab}). For $0 \leq i,j\leq d$ we multiply
each side of
(\ref{eq:dg1}) on the left by $E_i$ and on the right by $E_j$.
Simplify the result to get
\begin{eqnarray}
0 = E_i B E_j (\theta_i-\theta_j)\bigl(
\theta^2_i - (q^2+q^{-2})\theta_i \theta_j + \theta^2_j
+ (q^2-q^{-2})^2\bigr).
\label{eq:trip}
\end{eqnarray}
Now assuming $|i-j|=1$ and using
$E_i B E_j \not=0$,
\begin{eqnarray}
\label{eq:qreq}
0 = \theta^2_i - (q^2+q^{-2}) \theta_{i} \theta_j + \theta^2_j
+(q^2-q^{-2})^2.
\end{eqnarray}
Let $p(i,j)$ denote the right-hand side of
(\ref{eq:qreq}). 
For $1 \leq j \leq d-1$,
\begin{eqnarray*}
\theta_{j-1} - (q^2+ q^{-2}) \theta_j + \theta_{j+1} 
&=& \frac{p(j-1,j)-p(j,j+1)}{\theta_{j-1}-\theta_{j+1}}
\\
&=&  0.
\end{eqnarray*}
By the above recurrence there exist 
$u,v \in \mathbb F$ such that
\begin{eqnarray}
\label{eq:thForm}
\theta_i = u q^{d-2i} + v q^{2i-d} \qquad \qquad (0 \leq i \leq d).
\end{eqnarray}
Since $d\geq 1$ we have
the equation $0 = p(0,1)$. Evaluate this equation
using 
(\ref{eq:thForm}) 
 to obtain
$uv=1$. This yields the equation on the left in
(\ref{eq:ab}).
The equation on the right in
(\ref{eq:ab}) is similarly obtained.
\end{proof}

\begin{definition}\rm Define
\begin{eqnarray*}
t_i = a^{2i}q^{2i(d-i)} \qquad \qquad (0 \leq i \leq d).
\end{eqnarray*}
\end{definition}

\noindent The following calculation will be
useful.

\begin{lemma}
\label{lem:calc}
For $0 \leq i,j\leq d$ such that
$|i-j|\leq 1$, 
\begin{equation}
  \frac{t_j}{t_i}
=
1 + \frac{q\theta^2_i - (q+q^{-1})\theta_i \theta_j
+ q^{-1}\theta_j^2}{(q-q^{-1})(q^2-q^{-2})}
\label{eq:calc}
\end{equation}
and
\begin{equation}
\frac{t_i}{t_j}=
1 + \frac{q^{-1}\theta^2_i - (q+q^{-1})\theta_i \theta_j
+ q\theta_j^2}{(q-q^{-1})(q^2-q^{-2})}.
\label{eq:calc2}
\end{equation}
\end{lemma}
\begin{proof}
Each side of (\ref{eq:calc}) is equal to
$q^{4i-2d-2}a^{-2}$ (if $i-j=1$),
$1$ (if $i=j$), and
$q^{2d+2-4j}a^2$ (if $i-j=-1$).
Each side of (\ref{eq:calc2}) is equal to
$q^{2d+2-4i}a^2$ (if $i-j=1$),
$1$ (if $i=j$), and
$q^{4j-2d-2}a^{-2}$ (if $i-j=-1$).
\end{proof}

\begin{definition}
\label{def:psi}
\rm 
Define
\begin{eqnarray}
\label{eq:psi1}
\Psi = \sum_{i=0}^d t_i E_i.
\end{eqnarray}
\end{definition}

\begin{lemma} The map $\Psi$ is invertible, and
\begin{eqnarray}
\label{eq:psi2}
\Psi^{-1} = \sum_{i=0}^d t^{-1}_i E_i.
\end{eqnarray}
\end{lemma}
\begin{proof} Since 
$I = \sum_{i=0}^d E_i$ and
$E_i E_j = \delta_{i,j} E_i$ for $0 \leq i,j\leq d$.
\end{proof}

\begin{theorem}
\label{thm:Oq}
For $X \in {\mathcal O}_q$ the following holds on $V$:
 \begin{eqnarray}
L(X) = \Psi^{-1} X \Psi.
\label{eq:LPXP}
\end{eqnarray}
\end{theorem}
\begin{proof}
It suffices to show
$L(A) = \Psi^{-1} A \Psi$
and
$L(B) = \Psi^{-1} B \Psi$.
Certainly
$L(A) = \Psi^{-1} A \Psi$, since
$L(A)=A$ and 
$A$ commutes with $\Psi$.
We now verify 
$L(B) = \Psi^{-1} B \Psi$.
Since $I=\sum_{i=0}^d E_i$ it suffices to
show
$E_iL(B)E_j = E_i\Psi^{-1} B\Psi E_j$
for $0 \leq i,j\leq d$. Let $i,j$ be
given. Using the definition 
(\ref{eq:bk1})
of $L$, one finds that 
$E_iL(B)E_j$ is equal to
$E_iBE_j$ times the scalar on the
right in (\ref{eq:calc}).
Using Definition
\ref{def:psi} one finds that
$E_i\Psi^{-1} B \Psi E_j$ is equal to
$E_iBE_j$ times the scalar on the
left in 
(\ref{eq:calc}).
For the moment assume $|i-j|\leq 1$.
Then
$E_iL(B)E_j = E_i\Psi^{-1} B \Psi E_j$
by Lemma
\ref{lem:calc}. Next assume
 $|i-j|>1$.
Then $E_iL(B)E_j = E_i\Psi^{-1} B \Psi E_j$
since $E_iBE_j=0$.
We have shown
$L(B) = \Psi^{-1} B \Psi$.
\end{proof}

\noindent 
In Lemma
\ref{lem:calc} we related the
parameters
$\lbrace t_i\rbrace_{i=0}^d$ and
$\lbrace \theta_i\rbrace_{i=0}^d$.
We now give a more general result along this line.

\begin{theorem}
\label{lem:ssum}
For $0 \leq i,j\leq d$ we have
\begin{eqnarray*}
\frac{t_j}{t_i} = 1 
+ 
\sum_{n=1}^\infty
\Biggl(
\frac{\theta_i - \theta_j}{q-q^{-1}}
\prod_{r=1}^{n-1} 
\frac{ (q^{2r}-q^{-2r})^2 +
(q^r \theta_i - q^{-r} \theta_j)
(q^{-r} \theta_i - q^{r} \theta_j)
}{
(q^{2r}-q^{-2r})
(q^{2r+1}-q^{-2r-1})
}
\Biggr)
\frac{q^n \theta_i - q^{-n} \theta_j}{q^{2n}-q^{-2n}}
\end{eqnarray*}
and
\begin{eqnarray*}
\frac{t_i}{t_j} = 1 
+ 
\sum_{n=1}^\infty
\Biggl(
\frac{\theta_i - \theta_j}{q-q^{-1}}
\prod_{r=1}^{n-1} 
\frac{ (q^{2r}-q^{-2r})^2 +
(q^r \theta_i - q^{-r} \theta_j)
(q^{-r} \theta_i - q^{r} \theta_j)
}{
(q^{2r}-q^{-2r})
(q^{2r+1}-q^{-2r-1})
}
\Biggr)
\frac{q^{-n} \theta_i - q^{n} \theta_j}{q^{2n}-q^{-2n}}.
\end{eqnarray*}
Moreover, in the above sums the large parenthetical expression
is zero for $n> |i-j|$.
\end{theorem}
\begin{proof}
We verify the first displayed equation.
Since the $\mathcal O_q$-module $V$ is irreducible, there exists
$X \in \mathcal O_q$ such that $E_iXE_j\not=0$.
For this $X$, equation
(\ref{eq:LPXP}) holds on $V$.
In equation 
(\ref{eq:LPXP}),
multiply each side on the left by $E_i$ and 
on the right by $E_j$. Evaluate the results using
(\ref{eq:conj1}),
(\ref{eq:psi1}),
(\ref{eq:psi2}) together with $E_iXE_j \not=0$.
 This yields the first displayed equation
after a brief
calculation. The second displayed equation is similarly 
obtained using
$L^{-1}(X) = \Psi X \Psi^{-1}$.
The last assertion of the theorem statement
can be checked directly using
(\ref{eq:ab}).
\end{proof}

\begin{note} \rm It is natural to ask how we discovered
Theorem \ref{conj:T}. The answer is that we first discovered
Theorem \ref{lem:ssum}, and then considered the implications for $L$.
\end{note}

\section{Acknowledgment} The author thanks 
Pascal Baseilhac and Stefan Kolb for sharing their
preprint
\cite{BK}
prior to publication.

\bigskip

\noindent Paul Terwilliger \hfil\break
\noindent Department of Mathematics \hfil\break
\noindent University of Wisconsin \hfil\break
\noindent 480 Lincoln Drive \hfil\break
\noindent Madison, WI 53706-1388 USA \hfil\break
\noindent email: {\tt terwilli@math.wisc.edu }\hfil\break

  \end{document}